\documentclass{article}
\usepackage{graphicx} 
\usepackage{amsmath}
\usepackage{amssymb}
\usepackage{amsopn}
\usepackage{amsthm}
\usepackage{amsfonts}
\usepackage[right]{eurosym}
\usepackage{mathtools}
\usepackage{comment}
\usepackage{mathrsfs}
\usepackage{hyperref}
\usepackage[backend=biber, style=numeric]{biblatex}
\usepackage{authblk}
\usepackage{orcidlink}
\usepackage{color} 
\usepackage{graphicx}
\usepackage{tikz-cd}

\DeclareMathOperator{\bC}{\mathbb{C}}

\DeclareMathOperator{\bQ}{\mathbb{Q}}
\DeclareMathOperator{\bZ}{\mathbb{Z}}

\DeclareMathOperator{\bP}{\mathbb{P}}
\DeclareMathOperator{\cA}{\mathcal{A}}
\DeclareMathOperator{\cF}{\mathcal{F}}
\DeclareMathOperator{\cB}{\mathcal{B}}

\DeclareMathOperator{\cV}{\mathcal{V}}
\DeclareMathOperator{\cO}{\mathcal{O}}
\DeclareMathOperator{\Nu}{\mathcal{V}}

\DeclareMathOperator{\Sym}{Sym}
\DeclareMathOperator{\Pic}{Pic}
\DeclareMathOperator{\Kum}{Kum}
\DeclareMathOperator{\Div}{div}
\DeclareMathOperator{\gr}{\mathbf{Gr}}

\DeclareMathOperator{\disc}{disc}
\newcommand{\hilb}[2][n]{#2^{[#1]}}
\newcommand{\sym}[2][n]{#2^{(#1)}}

\newtheorem{proposition}{Proposition}[section]
\newtheorem{remark}[proposition]{Remark}
\newtheorem{lemma}[proposition]{Lemma}
\newtheorem{corollary}[proposition]{Corollary}
\newtheorem{theorem}[proposition]{Theorem}

\title{On Fujita's conjecture for a general hyperkaehler manifold in the standard series of examples}
\author{Alessandro Pilastro}
\affil{\small{Department of Mathematics\\
Stony Brook University \\
Stony Brook, NY 11794-3651, USA,\\
\url{alessandro.pilastro@stonybrook.edu}\\
{\textcolor{orcidlogocol}{\orcidlink{https://orcid.org/0009-0007-0242-8229}} orcid.org/0009-0007-0242-8229}}}
\date{}

\begin{document}

\maketitle
\begin{abstract}
   For the moduli spaces $\Nu(K3)_{d,t}^n$ and $\Nu(\Kum)_{d,t}^n$ of polarized hyperk\"aheler manifolds of ${K3}^{[n]}$-type and $\text{Kum}^n$-type respectively, with polarization with square $2d$ and divisibility $t$, we study general base point freeness and very ampleness of the polarization. We provide conditions, in terms of $n,d,t$ that guarantee that the general polarization, in any connected component of these moduli spaces, to be base point free or very ample.
    \bigskip
    
\end{abstract}

\section*{Introduction}

Fujita's conjecture states that, given an ample holomorphic line bundle $L$ on a compact complex manifold $X$ of complex dimension $n$, the line bundle $K_X\otimes L^{\otimes m}$ is base point free for all $m\geq n+1$ and very ample for $m\geq n+2$. In general, the conjecture is widely open, as only in low dimensions the conjecture is known to hold and in some specific examples of manifolds. Two important families that satisfy the conjecture are $K3$ surfaces and Abelian varieties, where the conjecture holds with lower bounds of $2$ and $3$, respectively. These values are stronger compared to what the conjecture suggests, so it would be interesting to study the conjecture on hyperk\"ahler manifolds.

Let $\Nu(K3)_{d,t}^n$ be the moduli space polarized of hyperk\"ahler manifolds of $\hilb{K3}$-type with polarization of square $2d$ and divisibility $t$; let $\Nu(\Kum)_{d,t}^n$ be the analogous moduli space for $\Kum^n$-type manifolds. In \cite{riess}, U.Rie{\ss} proves that the general element of $\Nu(K3)_{d,1}^2$ and $\Nu(K3)_{d,2}^2$ is base point free, while in \cite{Varesco_2023}, M.Varesco proves a similar statement for the $\Kum^n$-type for $n=2,3,4$ and all possible divisibilities, except for a finite number of cases. The results of U.Rie{\ss} are generalized by O.Debarre in \cite{debarre2020hyperkahlermanifolds}, where bounds, in terms of $d$ and $n$, are given for the general element of the moduli spaces $\Nu(K3)_{d,1}^n$ and $\Nu(K3)_{d,2}^n$ to be base point free and very ample. 

The idea of the cited works is to provide an element $(X,H)$ of the considered moduli spaces such that $H$ satisfied the required properties. Openness of these properties in families is then used to obtain the generality of the statements. The choices of $n$ and $t$ are made so that the moduli spaces have only one connected component, hence the previous argument can be applied. The same argument cannot immediately generalized to all the non empty moduli spaces $\Nu(K3)_{d,t}^n$ and $\Nu(\Kum)_{d,t}^n$, as a matter of fact in \cite{zbMATH06387271} and \cite{onorati2022connectedcomponentsmodulispaces} respectively, it is proven that these moduli spaces are not connected, and a count of the connected components is provided.

Inspired by these results, we formalize the techniques used in the aforementioned works to generalize the results in higher dimensions for all possible divisibility in both the case of $\hilb{K3}$-type and $\Kum^n$-type hyperk\"ahler. Since the case with divisibility $t=1$ is extensively covered in \cite{riess}, \cite{debarre2020hyperkahlermanifolds} (for the $\hilb{K3}$ case) and in \cite{Varesco_2023} (for the $\Kum^n$ case) we are going to be dealing with $t\geq 2$.

\begin{theorem}\label{theo_main}
   Let $n,d,t$ be positive integers with $n,t\geq 2$, let $\tau=\frac{t^2}{2(t-1)}$ and let $e_i=\frac{d+mi^2}{t^2}$.
     
     For the $\hilb{K3}$-type we have the following.\\
     Under the assumption $n\leq \frac{e_1}{2}+1$:
\begin{itemize}
    \item if $d\geq (\tau-1)n+\tau+1$ then the polarization of a general element in any connected component of $\Nu_{d,t}^n(K3)$ is base point free, and the $n+2$ power of the polarization is very ample;
    \item if $d\geq (\tau-1)n+2\tau+1$ then the polarization of a general element in any connected component of $\Nu_{d,t}^n(K3)$ is very ample.
\end{itemize}   
    Otherwise,
    \begin{itemize}
        \item if $d\geq \left(\frac{t^2}{2}-1\right)n+\frac{t^2}{2}+1$ then the polarization of a general element in any connected component of $\Nu_{d,t}^n(K3)$ is base point free, and the $n+2$ power of the polarization is very ample;
        \item if $d\geq \left(\frac{t^2}{2}-1\right)n+t^2+1$ then the polarization of a general element in any connected component of $\Nu_{d,t}^n(K3)$ is very ample.
    \end{itemize}
    
For the $\Kum^n$-type we have the following.\\
Under the assumption $n\leq \frac{e_1-3}{2}$:
\begin{itemize}
    \item if $d\geq (\tau-1)n+2\tau-1$ then the polarization of a general element in any connected component of $\Nu_{d,t}^n(\Kum)$ is base point free and the $n+2$ power of the polarization is very ample;
    \item if $d\geq (\tau-1)n+3\tau-1$ then the polarization of a general element in any connected component of $\Nu_{d,t}^n(\Kum)$ is very ample.
\end{itemize}
Otherwise,
\begin{itemize}
    \item if $d\geq \left(\frac{t^2}{2}-1\right)n+t^2-1$ then the polarization of a general element in any connected component of $\Nu_{d,t}^n(\Kum)$ is base point free, and the $n+2$ power of the polarization is very ample;
    \item if $d\geq \left(\frac{t^2}{2}-1\right)n+\frac{3t^2}{2}-1$ then the polarization of a general element in any connected component of $\Nu_{d,t}^n(\Kum)$ is very ample.
\end{itemize}
\end{theorem}

The idea is to use the maps that A.Apostolov and C.Onorati use in their works to realize an element of a prescribed connected component in $\Nu_{d,t}^n(\hilb{K3})$ and $\Nu_{d,t}^n(\Kum)$.

Still interesting is the situation where there is only one connected component, as the bounds are always linear in $n$ and $t$.

\begin{theorem}
    Let $n,t,d$ be positive integers with $n,t\geq 2$ and let $\tau:=\frac{t^2}{2(t-1)}$. Then the following holds:
    \begin{itemize}
        \item if $d\geq (\tau-1)n+\tau+1$ there exists a connected component of $\Sigma_{d,n}^t$ such that the polarization of a general element $(X,H)$ is base point free and $H^{n+2}$ is very ample;
        \item if $d\geq (\tau-1)n+2\tau+1$ there exists a connected component of $\Sigma_{d,n}^t$ such that the polarization of a general element is very ample;
        \item if $d\geq (\tau-1)n+2\tau-1$ there exists a connected component of $\Upsilon_{d,n}^t$ such that the polarization of a general element $(X,H)$ is base point free and $H^{n+2}$ is very ample;
        \item if $d\geq (\tau-1)n+3\tau-1$ there exists a connected component of $\Upsilon_{d,n}^t$ such that the polarization of a general element is very ample;
        These hold unless the considered moduli space is empty.
    \end{itemize}
\end{theorem}

Using the formulas in \cite{onorati2022connectedcomponentsmodulispaces} and \cite{zbMATH06387271}, we can provide som examples where the moduli spaces have one connected component and the better bound holds. With divisibilities $t=1,2$ and $t=p^a$, with $p>2$ prime, and $p^{a+1}$ not dividing $n$, the moduli spaces $\Nu(K3)_{d,t}^n$ and $\Nu(\Kum)_{d,t}^n$ are connected. We also provide, in Proposition \ref{prop_char_of_non_emptyness}, a characterization of when the considered moduli spaces are non-empty. This allows us to improve the inequalities presented in Theorem \ref{theo_main} in any specific case. 

The structure of the article is the following.

In Section \ref{sect_tautological}, we recall the main properties of the Hilbert scheme of points on a surface, with a focus on the tautological vector bundle $\hilb{L}$ on $\hilb{S}$ associated to any line bundle $L$ on the surface $S$. We prove how this construction preserves base point freeness. 

In Section \ref{sect_grass}, we provide the general notion of embeddings into Grassmannians via suitable vector bundles. In particular, we show how the tautological vector bundle $\hilb{L}$ defines a map into a Grassmannian if $L$ is $(n-1)$-very ample, which is an embedding if $L$ is $n$-very ample. Taking the determinant of $\hilb{L}$ allows us to obtain a base point free and a very ample line bundle on $\hilb{S}$.

Sections \ref{sect_hyper} and \ref{sect_moduli} recall the definition of hyperk\"ahler manifolds and provide the statements of some properties that we will use throughout the article. In particular, we present the lattice structure of the Picard group of $\hilb{K3}$-type and $\Kum^n$-type hyperk\"ahler manifolds. We also show why base point freeness and very ampleness are open properties in families for hyperk\"ahler manifolds.

In section \ref{sect_main_theorems}, we combine the results to obtain the main theorem. The main idea of the proof is to build an element of the moduli space with the required properties in the prescribed connected component. The openness of base point freeness and the very ampleness allows us to conclude that the general elements in the same connected component share the same properties. 

\section{Tautological vector bundles on Hilbert scheme of a surface}\label{sect_tautological}

Let $S$ be a smooth, projective, complex surface and let $\hilb{S}:=\text{Hilb}^n(S)$ the Hilbert scheme parameterizing $0-$dimensional subschemes of length $k$ of $S$. For an introduction on the subject one can refer, among other sources, to \cite{nakajima_hilbert_schemes}.  Since $\hilb{S}$ is a fine moduli space we have a universal family $$\Xi_k=\{(x,\xi)\in S\times \hilb{S}|x\in \xi\}\subset S\times \hilb{S}$$ together with the natural projections $$S\xleftarrow{\pi_S}\Xi_k\xrightarrow{\pi_{\hilb{S}}}\hilb{S}.$$
Given any coherent sheaf on $F$ on $S$ we can construct the so called \emph{tautological sheaf associated to $F$} on $\hilb{S}$ as follows: $$\hilb{F}:=\pi_{\hilb{S}*}(\pi^*_S(F)).$$
Since the projection $\pi_{\hilb{S}}$ is flat of degree $k$, if $F$ is a vector bundle of rank $m$,  then the associated tautological bundle $\hilb{F}$ is a vector bundle of rank $mk$ with fibers
\begin{align}\label{eq_fiber_of_F^[n]}
    \hilb{F}(\xi)\cong H^0(\xi,F|_\xi),\ \forall \xi\in\hilb{S}.
\end{align}

The tautological bundles have the following properties.

\begin{proposition}[\cite{bini2022familiesbigstablebundles}, (4.5)]\label{prop_zeroth_cohom_F^[n]}
Let $S$ be a smooth, irreducible projective surface and let $F$ be a coherent sheaf on $S$ then $H^0(\hilb{S},\hilb{F})\cong H^0(S,F)$.
\end{proposition}

Recall that $F$ is $(n-1)$-\emph{very ample} on $S$ if the natural evaluation map $$H^0(S,F)\to H^0(S,F\otimes \cO_\xi)$$ is surjective for all $0$-dimensional subschemes $\xi$ of length $n$.

\begin{proposition}[\cite{bini2022familiesbigstablebundles}, Prop 4.5]\label{prop_tautological_globally_generated}
Let $S$ be a smooth, irreducible projective surface and let $F$ be a rank $r$ vector bundle on $S$ with $r\geq 1$. Let $n\geq 2$ be an integer. Then $F$ is $(n-1)$-very ample on $S$ if and only if the tautological bundle $\hilb{F}$ is globally generated.
\end{proposition}

Denoting the $k$-th cartesian product $S^k=S\times\dots\times S$ and the $k$-th symmetric product $\sym{S}=S^n/\Sym(n)$, we have the \emph{Hilbert-Chow morphism} $$\mu:\hilb{S}\to\sym{S}$$ given by $\mu(\xi)=\sum_{x\in\text{Supp($\xi$)}}\text{ length}_\xi(x)x$.
Let $L\in \Pic(S)$, denoting by $\pi_i:S^n\to S$ the $i$-th natural projection, we can define the line bundle $L^{\boxtimes n}:=\otimes_{i=1}^n\pi^*_i(L)$ which descends to a line bundle $\sym{L}\in\Pic(\sym{S})$. This provides a morphism $D_n:\Pic(S)\to \Pic(\hilb{S})$ by $D_n(L)=L_n$ where $L_n:=\mu^*(\sym{L})$.    

Using this morphism, we can explicitly write what the determinant line bundle of a tautological vector bundle is. Let $E$ be the exceptional divisor of the Hilbert-Chow morphism, and let $\delta\in  \Pic(\hilb{S})$ such that $2\delta=E$.

\begin{proposition}[\cite{Lehn_1999}, Lemma 3.7 - Thm 4.6]\label{prop_det_of_tautological}
    Let $S$ be a smooth surface and let L be a line bundle on $S$, then the line bundle $\det \hilb{L}=L_n-\delta$ on $\hilb{S}$.
\end{proposition}

To conclude this section, let us show that the morphism $D_n$ preserves base point freeness.

\begin{lemma}\label{lemma_bpf_of_L_n_on_Hilb_n}
Let $S$ be a smooth projective complex surface, and let $L\in\Pic(S)$ be base point free, then $L_n\in \Pic(\hilb{S})$ is also base point free.    
\end{lemma}
\begin{proof}
Let $\xi\in\hilb{S}$ and let $x_1,\dots,x_n$ be the points in the support of $\xi$ with multiplicities. We can find a section $s\in H^0(S,L)$ such that $s(x_i)\neq 0$ for all $i=1,\dots,n.$ This follows by induction on $n$. The base case is the definition of base point free line bundle. From the induction hypothesis, we have a section $\hat{s}$ such that $\hat{s}(x_i)\neq 0$ for $i=1,\dots,n-1$. If $\hat{s}(x_n)\neq 0$ then we are done. Suppose then that $\hat{s}(x_n)= 0$ and given $\tilde{s}$ a section non-vanishing at the point $x_n$, define $s=\alpha \hat{s}+\tilde{s}$. For a general $\alpha\in\bC$ the section $s$ satisfies the required property.

So, given $\xi\in\hilb{S}$ let $s$ be a section of $L$ non-vanishing at the points in the support of $\xi$. Consider now the section $\otimes_{i=1}^n\pi_i^*(s)\in H^0(S^n,L^{\boxtimes n})$ which is invariant under the action of $\Sym(n)$. It therefore descends to a section $\sym{s}\in H^0(\sym{S},\sym{L})$. Then $\sym{s}$ takes the value $\prod_{i=1}^ns(x_i)\neq0$ at the point $\sum_{i=1}^nx_i\in\sym{S}$. The pullback section $\mu^*(\sym{s})=\hilb{s}\in H^0(\hilb{S},L_n)$ does not vanish at the point $\xi\in\hilb{S}$ as required.

\end{proof}

 \section{Embeddings of Hilbert schemes into Grassmannians}\label{sect_grass}

We start by setting some notations for the Grassmannians. Let $V$ be a complex $d$-dimensional vector space, with $\gr(V,n)$ we denote the space of $n$-dimensional subspaces of $V$. By $U$ we denote the universal subbundle and by $Q$ the universal quotient bundle, which naturally fits in the following short exact sequence $$0\to U\to \gr(V,n)\times V\to Q\to 0.$$ 

Let $X$ be a compact complex manifold, and let $B$ be a globally generated holomorphic vector bundle of rank $n$ on $X$. As in the case of line bundles, $G$ induces a map $\Phi_B$ to $\gr(V,n)$, where $V=H^0(X,B)^*$, as follows

\begin{align*}
    \Phi_B: X&\to \gr(V,n)\\
    p&\mapsto (\ker v_p)^\circ
\end{align*}
where $v_p:H^0(X,B)\to B_p$ is the evaluation map at the point $p\in X$, and $(\ker v_p)^\circ$ denotes the annihilator of $\ker v_p$. Note that the map is well defined since $B$ is globally generated. As a last general remark, consider the following composition 
$$X\xrightarrow{\Phi_B}\gr(V,n)\xhookrightarrow{P} \bP(\Lambda^n V)$$ where the last map is the Pl\"ucker embedding, i.e. the embedding induced from the line bundle $\det Q$. It holds that $\Phi_B^*Q\cong B$, hence $\det B\cong(\Phi_B\circ P)^*\cO_{\bP(\Lambda^nV)}(1)$. 
From this observation, we obtain the following:

\begin{lemma}\label{lemma_vector_bundle_to_determinat_line_bundle}
Let $X$ be a compact complex manifold, and let $B$ be a globally generated vector bundle on X, then $\det B$ is base point free. Furthermore, if $\Phi_B$ is an embedding, then $\det B$ is very ample.
\end{lemma}

Returning to the Hilbert schemes, let $S$ be a complex compact surface and let $L$ be a $(n-1)$-very ample line bundle on $S$. Proposition \ref{prop_tautological_globally_generated} guarantees that the vector bundle $\hilb{L}$ is globally generated, hence we have the well defined map 
\begin{align*}
    \Phi_{\hilb{L}}:\hilb{S}&\to \gr(H^0(\hilb{S},\hilb{L})^*,n).\\
    \eta &\mapsto (\ker v_\eta)^\circ
\end{align*}

In this case, we have another map into a Grassmannian. As a matter of fact, recall that $(n-1)-$generation of $L$ implies that for all $\eta\in \hilb{S}$ the restriction $r_\eta :H^0(S,L)\to H^0(\eta,L\otimes\cO_\eta)$ is surjective, in particular this allows us to define a map 

\begin{align*}
    \varphi_{\hilb{L}}:\hilb{S}&\to \gr(H^0(S,L)^*,n).\\
    \eta &\mapsto (\ker r_\eta)^\circ 
\end{align*}

In \cite{Catanese_Goettsche}, the following result is proven:

\begin{theorem}
The morphism $\varphi_{\hilb{L}}:\hilb{S}\to \gr(H^0(S,L)^*,n)$ is an embedding if and only if $L$ is $n$-very ample.
\end{theorem}

Consider now the map 
\begin{align*}
    g:H^0(S,L)\to H^0(\hilb{S},\hilb{L})
\end{align*} 
where $g(s)(\eta)=s|_\eta\in H^0(\eta,\cO_\eta\otimes L)$ for all $\eta\in \hilb{S}$. The map $g$ is injective and, thanks to Proposition \ref{prop_zeroth_cohom_F^[n]}, it is an isomorphism, hence it induces an isomorphism $$\tilde{g} :\gr(H^0(S,L)^*,n)\to \gr(H^0(\hilb{S},\hilb{L})^*,n).$$ Observe now that $r_\eta=v_\eta\circ g$ for all $\eta\in\hilb{S}$, hence the following diagram is commutative 

\[\begin{tikzcd}
	{\hilb{S}} & {\gr(H^0(S,L)^*,n)} \\
	& {\gr(H^0(\hilb{S},\hilb{L})^*,n)} \\
	& {}
	\arrow["{\varphi_{\hilb{L}}}", from=1-1, to=1-2]
	\arrow["{\Phi_{\hilb{L}}}"', from=1-1, to=2-2]
	\arrow["{\tilde{g}}", from=1-2, to=2-2]
\end{tikzcd}\]

that is $\Phi_{\hilb{L}}=\tilde{g}\circ\varphi_{\hilb{L}}$.

\begin{proposition}\label{prop_det_of_tautological_bpf_va}
    Let $S$ be a compact complex smooth surface and let $L$ line bundle on $S$, the following hold:
    \begin{itemize}
        \item If $L$ is $(n-1)$-very ample then $\det \hilb{L}$ is base point free;
        \item If $L$ is $n$-very ample then $\det \hilb{L}$ is very ample.
    \end{itemize}
    In both cases the same property is shared with the line bundle $L_n-\delta$.
\end{proposition}

\begin{proof}
If $L$ is $(n-1)$-very ample, Proposition \ref{prop_tautological_globally_generated} implies that $\hilb{L}$ is globally generated, then Lemma \ref{lemma_vector_bundle_to_determinat_line_bundle} implies that $\det \hilb{L}$ is base point free as wanted.

Suppose now that $L$ is $n$-very ample. Under this assumption, the map $\Phi_{\hilb{L}}=\tilde{g}\circ\varphi_{\hilb{L}}$ is a composition of two embeddings, hence an embedding itself. We can conclude using Lemma \ref{lemma_vector_bundle_to_determinat_line_bundle}.

The last observation follows immediately from Proposition \ref{prop_det_of_tautological}.
\end{proof}

\section{Hyperk\"ahler manifolds}\label{sect_hyper}

Let $X$ be a K\"ahler manifold, we say that $X$ is \emph{hyperk\"ahler} (or \emph{irreducible holomorphic symplectic}) if it is simply connected and if $H^0(X,\Omega^2_X)$ is generated by an everywhere non-degenerate holomorphic 2-form.
There are two known families of examples that occur in every possible dimension:
\begin{itemize}
    \item Given a $K3$ surface $S$ its Hilbert scheme on $n$ points $\hilb[n]{S}$ is a hyperk\"ahler manifold;
    \item  Given an abelian surface $A$ consider the Hilbert scheme $\hilb[n+1]{A}$. There is a natural map $\hilb[n+1]{A}\to A$ given by the sum. The fiber over $0$ is a hyperk\"ahler manifold called \emph{generalized Kummer variety}, denoted by $\Kum^n(A)$.
\end{itemize}

Furthermore, deformations of hyperk\"ahler manifolds are still hyperk\"ahler. In particular, deformations of $\hilb[n]{S}$ are called of $\hilb[n]{K3}$-type, while deformations of $\Kum^n(A)$ are called of $\Kum^n$-type.
The only other known examples of hyperk\"ahler manifolds are O'Grady's "exceptional" examples, namely OG6 and OG10, which occur only in dimensions 6 and 10, respectively. 

Given an hyperk\"ahler manifold, the second cohomology group $H^2(X,\bZ)$ can be endowed with a quadratic form $q$ called the Beauville-Bogomolov-Fujiki form (or BFF form), so that $(H^2(X,\bZ),q)$ has the structure of a lattice (see \cite{huybrechts1997compacthyperkaehlermanifoldsbasic} for a reference). In this regard, let us recall some definitions. Let $(\Lambda,q)$ a lattice:
\begin{itemize}
    \item we say an element $\alpha$ is \emph{primitive} if $\alpha=k\alpha'$ implies $k=\pm1$;
    \item The \emph{multiplicity} of an element $\alpha$ is $m$ if there exists a primitive element $\alpha'$ such that $\alpha=m\alpha'$;
    \item The \emph{divisibility} of an element $\alpha$, denoted by $\Div(\alpha)$, is the multiplicity of $(\cdot,\alpha)_q$ as an element of the dual lattice $\Lambda^\vee$. Equivalently $\Div(\alpha)=m$ if $m$ generates the ideal $\{(\alpha',\alpha)_q|\alpha'\in\Lambda\}$;
    \item Let $\{e_i\}$ be a $\bZ$-basis of the lattice $\Lambda$. The \emph{discriminant} of the $\Lambda$ is defined as the determinant of the matrix $(q(e_i,e_j))_{i,j}$.
\end{itemize}

We report a characterization of projective hyperk\"ahler manifolds.
\begin{theorem}[\cite{huybrechts1997compacthyperkaehlermanifoldsbasic}, Theorem 3.11]\label{theo_projective_hyperkahler_manifolds}
    Let $X$ be a hyperk\"ahler manifold. Then $X$ is projective if and only if there exists a line bundle $L$ on $X$ with $q(L)>0$.
\end{theorem}

The lattice structure of the $\hilb[n]{K3}$-type and $\Kum^n$-type varieties is described in the following proposition. Fix the notation $\Lambda_{K3}\cong U^{\oplus 3}\oplus E_8(-1)^{\oplus 2}$ for the $K3$ lattice, where $U$ is the standard hyperbolic lattice and $E_8(-1)$ is the unimodular root-lattice $E_8$ change by a sign.

\begin{proposition}\label{p_k3_type_lattice}
    Fix a $K3$ surface $S$. Then the following holds:
    \begin{itemize}
        \item[(i)] For every $n\geq 2$, there is an orthogonal decomposition of lattices $$H^2(\hilb[n]{S},\bZ)\cong H^2(S,\bZ)\oplus \bZ\delta,$$
        where $\delta$ is an integral $(1,1)$-class with $q(\delta)=-2(n-1)$. The class $2\delta$ is represented by the Hilbert-Chow divisor.
        \item[(ii)] There is an orthogonal decomposition:
        $$\Pic(\hilb[n]{S})\cong\Pic(S)\oplus\bZ\delta.$$
        \item[(iii)]For each irreducible symplectic variety $X$ of $\hilb[n]{K3}$-type there exists an isomorphism $$H^2(X,\bZ)\cong\Lambda_{K3}\oplus \bZ\delta.$$  
    \end{itemize}
\end{proposition}

\begin{proposition}\label{p_kum_type_lattice}
    Fix an abelian surface $A$. Then the following holds:
    \begin{itemize}
        \item[(i)] For every $n\geq 2$, there is an orthogonal decomposition of lattices $$H^2(\Kum^n(A),\bZ)\cong H^2(A,\bZ)\oplus \bZ\delta,$$
        where $\delta$ is an integral $(1,1)$-class with $q(\delta)=-2(n+1)$. The class $2\delta$ is represented by the Hilbert-Chow divisor.
        \item[(ii)] There is an orthogonal decomposition:
        $$\Pic(\Kum^n(A))\cong\Pic(A)\oplus\bZ\delta.$$
        \item[(iii)]For each irreducible symplectic variety $X$ of $\Kum^n$-type there exists an isomorphism $$H^2(X,\bZ)\cong H^2(A,\bZ)\oplus \bZ\delta$$
        for some abelian variety $A$.
    \end{itemize}
\end{proposition}

\begin{remark}\label{remark_how_to_compute_divisibility}
Proposition \ref{p_k3_type_lattice} implies that given any element $\alpha\in H^2(X,\bZ)$ where $X$ is a hyperk\"ahler manifold of $\hilb[n]{K3}$-type, then $\alpha=a\lambda+b\delta$ with $\lambda\in \Lambda_{K3}$ a primitive element. In particular, recalling that $\Lambda_{K3}$ is a unimodular lattice we get $$\Div(\alpha)=\gcd(a,2b(n-1)).$$ 
Analogously for $X$ $\Kum^n$-type we have $$\Div(\alpha)=\gcd(a,2b(n+1)),$$
since $H^2(T,\bZ)$ is unimodular for any abelian surface $T$.
\end{remark}

\section{Moduli space of polarized hyperk\"ahler manifolds}\label{sect_moduli}

In \cite{zbMATH06387271} and \cite{onorati2022connectedcomponentsmodulispaces}, the authors give a characterization of the number of connected components of the moduli space of polarized hyperk\"ahler manifolds of $\hilb[n]{K3}$ and $\Kum^n$ type, respectively. Since we will need their construction, we are going to give a short overview of their results. The base idea is the same for both hyperk\"ahler types, so we will provide a uniform argument. 

Let $S$ be a K3 surface or an Abelian surface and fix $m=n-1$ in the $\hilb[n]{K3}$ case and $m=n+1$ in the $\Kum^n(S)$ case. Denote by $\Lambda_n$ the lattice $E_8(-1)^{\oplus 2}\oplus U^{\oplus 3}\oplus\langle-2(n-1)\rangle$ in the $\hilb[n]{K3}$ case or the lattice $U^{\oplus 3}\oplus\langle -2(n+1)\rangle$ in the $\Kum^n(S)$ case, where $E_8(-1)$ is the rank 8 negative definite $E_8$ root lattice, $U$ is the rank 2
hyperbolic lattice, and $\langle-2m\rangle$ is a rank one lattice generated by an element of square $-2m$. With this definition $\Lambda_n$ is isomorphic to $(H^2(X,\bZ),q)$ when $X$ is of $\hilb{K3}$ or $\Kum^n$ type. Denote by $\widetilde{\Lambda}$ the lattice $H^0(S,\bZ)\oplus H^2(S,\bZ)\oplus H^4(S,\bZ)$ endowed with the quadratic form $$(\alpha_1,\alpha_2,\alpha_3)=\int_S\alpha_2^2-2\alpha_1\alpha_3,$$ this is known as the Mukai lattice. Abstractly, $\widetilde{\Lambda}\cong E_8(-1)^{\oplus 2}\oplus U^{\oplus 4}$ in the $K3$ case and $\widetilde{\Lambda}\cong U^{\oplus 4}$ in the Abelian surface case.

Recall the following theorem, see \cite{markman_survey_torelli_and_monodromy}, Corollary 9.5, and \cite{wieneck_kummer}, Theorem 4.9, for the $\hilb{K3}$ and $\Kum^n$ type, respectively.

\begin{theorem}\label{theorem_distinguished_isometric_embedding_orbit}
    Let $X$ be a hyperk\"ahler manifold of $\hilb{K3}$ or $\Kum^n$ type with $n\geq 2$. Then $X$ comes with a natural choice of an $O(\widetilde{\Lambda})$-orbit $[\iota_X]$ of primitive isometric embeddings of $H^2(X,\bZ)$ into the Mukai lattice $\widetilde{\Lambda}$.
\end{theorem}

A. Apostolov and C.Onorati used these results to obtain a lattice-theoretic characterization of polarized parallel-transport operators.

Choose a primitive isometric embedding $\iota_X:H^2(X,\bZ)\hookrightarrow \widetilde{\Lambda}$ in the $O(\widetilde{\Lambda})$-orbit given by the previous theorem. For a primitive class $h\in H^2(X,\bZ)$ with square $2d>0$, denote by $T(X,h)$ the saturation of the lattice in $\widetilde{\Lambda}$, generated by $\iota_X(h)$ and $(\iota_X(H^2(X,\bZ)))^\perp$. With this definition $T(X,h)$ is then an even rank 2 positive definite lattice.

Define the sets $\Sigma^n$ and $\Upsilon^n$ to be the sets of isometric classes of pairs $[(T,h)]$ where $T$ is an even positive definite rank 2 lattice, and $h$ a primitive element such that $h^\perp =\langle-2(n-1)\rangle$ or $h^\perp =\langle-2(n+1)\rangle$ respectively. We say $(T,h)$ is isometric to $(T',h')$ if and only if there exists an isometry $\gamma: T\to T'$ sending $h$ to $h'$. Denote by $\Sigma_{d,t}^n$ and $\Upsilon_{d,t}^n$ the subsets of $\Sigma^n$ and $\Upsilon^n$ for which $T$ has discriminant $\frac{4dm}{t^2}$ and $h$ has degree $2d>0$.

Let $\mu(K3)_{d,t}^n$ the set of connected components of $\Nu(K3)_{d,t}^n$ parametrizing polarized hyperk\"ahler manifold of $\hilb{K3}$ with polarization of degree $2d$ and divisibility $t$. Consider the map $f:\mu_{d,t}^n(K3)\to \Sigma_{d,t}^n$ given by mapping a connected component $\mathcal{V}^0$ of $\Nu(K3)_{d,t}^n$ to $[T(X,h),\iota_X(h)]$ where $[(X,h)]\in \mathcal{V}^0$. An analogous map $f:\mu(\Kum)_{d,t}^n\to \Upsilon_{d,t}^n$ can be defined. The authors of \cite{zbMATH06387271} and \cite{onorati2022connectedcomponentsmodulispaces} prove that this map is a bijection in the $\hilb{K3}$ and $\Kum^n$ case respectively (\cite{zbMATH06387271}, Corollary 2.4 and \cite{onorati2022connectedcomponentsmodulispaces}, Proposition 2.5).

\begin{proposition}
    The maps $f:\mu(K3)_{d,t}^n\to\Sigma_{d,t}^n$ and $f:\mu(\Kum)_{d,t}^n\to\Upsilon_{d,t}^n$, defined as above, are bijections.
\end{proposition}

We also report the theorems counting the connected components of $\Sigma_{d,n}^t$ and $\Upsilon_{d,n}^t$. In the following, write $\phi$ for the Euler function, and let $\rho(a)$ denote the number of distinct primes in the factorization of a positive integer $a$.

\begin{theorem}[\cite{zbMATH06387271}, Prop 3.1]\label{theo_count_connected_components_hilb_k3}
    Let $t$ be a divisor of $(2d, 2n-2)$, and set 
    \begin{equation}
\begin{array}{lll}
\tilde{d}=\frac{2d}{\gcd(2d,2n-2)}, & \tilde{n}=\frac{2n-2}{\gcd(2d,2n-2)}, & g=\frac{\gcd(2d,2n-2)}{t} \\
 & & \\
w=\gcd(g,t), & g_1=\frac{g}{w}, & t_1=\frac{t}{w}.
\end{array}
\end{equation}
Put $w=w_+(t_1)w_-(t_1)$ where $w_+(t_1)$ is the product of all powers of primes dividing $\gcd(w, t_1)$.

     \begin{itemize}
       \item $|\Sigma_{d,t}^n|=w_+(t_1)\phi(w_-(t_1))2^{\rho(t_1)-1}$ if $t>2$ and one of the following sets of conditions hold: 
       \begin{itemize}
           \item[(i)] $g_1$ is even, $(\tilde{d}, t_1)=(\tilde{n}, t_1)=1$ and the residue class $-\tilde{d}/\tilde{n}$ is a quadratic residue modulo $t_1$;
           \item[(ii)] $g_1, t_1$, and $\tilde{d}$ are odd, $(\tilde{d}, t_1)=(\tilde{n}, 2t_1)=1$ and $-\tilde{d}/\tilde{n}$ is a quadratic residue modulo $2t_1$;
           \item[(iii)] $g_1, t_1$, and $w$ are odd, $\tilde{d}$ is even, $(\tilde{d}, t_1)=(\tilde{n}, 2t_1)=1$ and $-\tilde{d}/(4\tilde{n})$ is a quadratic residue modulo $t_1$;
       \end{itemize}
       \item $|\Sigma_{d,t}^n|=w_+(t_1)\phi(w_-(t_1))2^{\rho(t_1/2)-1}$ if $t>2$, $g_1$ is odd, $t_1$ is even, $(\tilde{d}, t_1)=(\tilde{n}, 2t_1)=1$ and $-\tilde{d}/\tilde{n}$ is a quadratic residue modulo $2t_1$;
       \item $|\Sigma_{d,t}^n|=1$ if $t\leq 2$ and one of the following sets of conditions hold: 
       \begin{itemize}
           \item[(i)] $g_1$ is even, $(\tilde{d}, t_1)=(\tilde{n}, t_1)=1$ and the residue class $-\tilde{d}/\tilde{n}$ is a quadratic residue modulo $t_1$;
           \item[(ii)] $g_1, t_1$, and $\tilde{d}$ are odd, $(\tilde{d}, t_1)=(\tilde{n}, 2t_1)=1$ and $-\tilde{d}/\tilde{n}$ is a quadratic residue modulo $2t_1$; 
           \item[(iii)] $g_1, t_1$, and $w$ are odd, $\tilde{d}$ is even, $(\tilde{d}, t_1)=(\tilde{n}, 2t_1)=1$ and $-\tilde{d}/(4\tilde{n})$ is a quadratic residue modulo $t_1$;
           \item[(iv)] $g_1$ is odd, $t_1$ is even, $(\tilde{d}, t_1)=(\tilde{n}, 2t_1)=1$ and $-\tilde{d}/\tilde{n}$ is a quadratic residue modulo $2t_1$;
       \end{itemize}
       \item $|\Sigma_{d,t}^n|=0$, else.
     \end{itemize}
\end{theorem}

\begin{theorem}[\cite{onorati2022connectedcomponentsmodulispaces}, Thm 5.3]\label{theo_count_connected_components_kumm_type}
Let $t$ be a divisor of $\gcd(2d,2n+2)$. Consider the following notations,
\begin{equation}
\begin{array}{lll}
d_1=\frac{2d}{\gcd(2d,2n+2)}, & n_1=\frac{2n+2}{\gcd(2d,2n+2)}, & g=\frac{\gcd(2d,2n+2)}{t} \\
 & & \\
w=\gcd(g,t), & g_1=\frac{g}{w}, & t_1=\frac{t}{w}.
\end{array}
\end{equation}
For $w$ and $t_1$ as defined above, we write $w=w_+(t_1)w_-(t_1)$, where $w_+(t_1)$ is the product of all powers of the primes that appear in the factorization of $w$ and that divide $\gcd(w,t_1)$. We have:
\begin{itemize}

\item $|\Upsilon_{d,t}^n|=w_+(t_1)\phi(w_-(t_1))2^{\rho(t_1)-1}$ if $t>2$ and one of the following holds:
 \begin{itemize}
 \item $g_1$ is even, $\gcd(d_1,t_1)=1=\gcd(n_1,t_1)$ and $-d_1/n_1$ is a quadratic residue mod~$t_1$;
 \item $g_1,t_1$ and $d_1$ are odd, $\gcd(d_1,t_1)=1=\gcd(n_1,2t_1)$ and $-d_1/n_1$ is a quadratic residue mod~$2t_1$;
 \item $g_1,t_1$ and $w$ are odd, $d_1$ is even, $\gcd(d_1,t_1)=1=\gcd(n_1,2t_1)$ and $-d_1/4n_1$ is a quadratic residue mod~$t_1$.
 \end{itemize}

\item $|\Upsilon_{d,t}^n|=w_+(t_1)\phi(w_-(t_1))2^{\rho(t_1/2)-1}$ if $t>2$, $g_1$ is odd, $t_1$ is even, $\gcd(d_1,t_1)=1=\gcd(n_1,2t_1)$ and $-d_1/n_1$ is a quadratic residue mod~$2t_1$.

\item $|\Upsilon_{d,t}^n|=1$ if $t\leq2$ and one of the following holds:
\begin{itemize}
 \item $g_1$ is even, $\gcd(d_1,t_1)=1=\gcd(n_1,t_1)$ and $-d_1/n_1$ is a quadratic residue mod~$t_1$;
 \item $g_1,t_1$ and $d_1$ are odd, $\gcd(d_1,t_1)=1=\gcd(n_1,2t_1)$ and $-d_1/n_1$ is a quadratic residue mod~$2t_1$;
 \item $g_1,t_1$ and $w$ are odd, $d_1$ is even, $\gcd(d_1,t_1)=1=\gcd(n_1,2t_1)$ and $-d_1/4n_1$ is a quadratic residue mod~$t_1$;
 \item $g_1$ is odd, $d_1$ is even, $\gcd(d_1,t_1)=1=\gcd(n_1,2t_1)$ and $-d_1/n_1$ is a quadratic residue mod~$2t_1$.
 \end{itemize}

\item $|\Upsilon_{d,t}^n|=0$ otherwise.
\end{itemize}
\end{theorem}

Being interested in studying the base point freeness and very ampleness of general elements of these moduli spaces, we now investigate the openness condition of these properties. We start by recalling the following.

\begin{proposition}[\cite{huybrechts1997compacthyperkaehlermanifoldsbasic}, 1.14]\label{prop_irreducibility_connected_components}
    For any $n,d,t$ positive integers with $n>1$ the moduli spaces $\Nu(K3)_{d,t}^n$ and $\Nu(\Kum)_{d,t}^n$ are smooth. In particular every connected component of these moduli spaces is irreducible.
\end{proposition}

Using the notation in \cite{nugent2025moduliamplelinebundles}, let $f:X\to Y$ be a morphism and $L$ be a line bundle on $X$, define $$\cB\cF(f,L):\left\{y\in Y: L_y \text{ is base point free}\right\}$$
and 
$$\cV\cA(f,L):\left\{y\in Y: L_y \text{ is very ample}\right\}.$$

Then the following theorem holds.

\begin{theorem}[\cite{nugent2025moduliamplelinebundles}, Lemma 7.3 - Thm 7.4]\label{theorem_openness_bpf_va}
    Let $f:X\to Y$ be a proper flat morphism of locally noetherian scheme and let L be a line bundle on $X$ such that $h^0(X_y,L_y)$ is constant for all $y\in Y$. Then $\cB\cF(f,L)$ and $\cV\cA(f,L)$ are open. 
\end{theorem}

Specilizing to hyperk\"ahler manifolds we obtain the following.

\begin{corollary}\label{cor_openess_of_bpf_va}
Let $f:X\to S$ be a family of hyperk\"ahler manifolds over a connected base $S$. Let $L$ be a line bundle on $X$ such that $q(L_s)>0$ for all $s\in S$. Then:
\begin{itemize}
    \item if $L_0$ is base point free for some $0\in S$ then $L_s$ is base point for all $s$ in an open neighborhood of $0\in S$;
    \item if $L_0$ is very ample for some $0\in S$ then $L_s$ is very ample for all $s$ in an open neighborhood of $0\in S$;
\end{itemize}
\end{corollary}
\begin{proof}
    We follow the proof of Proposition 2.6 in \cite{Varesco_2023} which we report here for convenience. We consider the situation where $L_0$ is very ample as the other follows analogously. Since $q(L_0)>0$ then $L_0$ is big and nef by the Beauville-Fujiki relations [\cite{Gross2003CalabiYau}, Prop. 23.14], hence $h^i(X_0,L_0)=0$ for $i>0$ thanks to the Kodaira vanishing theorem. The semicontinuity theorem then implies that $h^i(X_s,L_s)=0$ for all $i>0$ for all $s$ in a neighborhood $U$ of $0\in S$. By flatness of $f$, the Euler characteristic $\chi(L_s)$ does not depend on $s$, therefore $h^0(X_s,L_s)$ is constant for all $s\in U$. By theorem \ref{theorem_openness_bpf_va}, the set $\cV\cA(f,L)$ is open in $U$ hence in $S$ as wanted.
\end{proof}

\begin{lemma}\label{lemma_bfp_does_not_require_ampleness}
Let $d,n,t$ be positive integers with $n\geq 2$, let $X$ be an hyperk\"ahler manifold of $\hilb{K3}$ type and let $L\in \Pic(X)$ be base point free such that $q(L)=2d$, $\Div(L)=t$. Then the general element in a connected component of $\Nu(K3)_{d,t}^n$, where $\dim_{\bC}(X)=n$,  has base point free polarization. The same holds for $\Nu(\Kum)_{d,t}^n$.
\end{lemma}
\begin{proof}
    A very general deformation of $(X,L)$, call it $(X_A,L_A)$, has Picard rank equal to 1 hence, by the Projectivity criteria in Theorem \ref{theo_projective_hyperkahler_manifolds}, the line bundle $L_A$ is ample since $q(L_A)>0$. The properties of being ample and base point free are open conditions hence, we can find a generic deformation $(X_{AF},L_{AF})$ such that $L_{AF}$ is ample, base point free and satisfy the same numerical properties of $L$. Consequently, the element $(X_{AF},L_{AF})$ defines an element of $\Nu(K3)_{d,t}^n$. We can conclude that the generic element of $\Nu(K3)_{d,t}^n$, in the same connected component as $(X,L)$, has base point free polarization since the property is open in families by Corollary \ref{cor_openess_of_bpf_va} and the connected components are irreducible by Proposition \ref{prop_irreducibility_connected_components}.
\end{proof}

\section{Base point freeness and very ampleness of a general element}\label{sect_main_theorems}
\subsection{Case 1: one connected component}
The goal of this section is to give a lower bound, in terms of $d,n,t$, that guarantees the existence of a connected component in $\Nu(K3)_{d,t}^n$ and $\Nu(\Kum)_{d,t}^n$ such that for the general element $(X,H)$ the line bundle $H$ is base point free or very ample. This is only really useful if the moduli spaces have only one connected component, so we will provide some explicit examples when this happens. The techniques applied in this section have been used in \cite{debarre2020hyperkahlermanifolds} and \cite{Varesco_2023} to tackle some cases in small divisibility. We utilize the same techniques for any triple $d,n,t$.

Let us start by recalling the following: for K3 and Abelian surfaces, there is a characterization for $k$-very ampleness that can be applied in the case the Picard rank is 1.

\begin{proposition}[\cite{debarre2020hyperkahlermanifolds}, Cor. 3.8]\label{prop_k_very_ample_k3}
Let $S$ be a K3 surface such that $\Pic(S)\cong\bZ L$ with $L$ ample and $L^2=2e$. Then the line bundle $L^{\otimes a}$ is $k$-very ample if and only if either
\begin{itemize}
    \item $a=1$ and $k\leq \frac{e}{2}$;
    \item $a\geq 2$ and $k\leq 2(a-1)e-2$.
\end{itemize}
\end{proposition}

\begin{proposition}[\cite{Alagal_2016}, Theo. 4.3 - Prop. 1.5]\label{prop_k_very_ample_abelian}
Let $A$ be an Abelian surface such that $\Pic(A)\cong\bZ L$ with $L$ ample and $L^2=2e$. Then the line bundle $L^{\otimes a}$ is $k$-very ample if and only if either
\begin{itemize}
    \item $a=1$ and $k\leq \lfloor\frac{e-3}{2}\rfloor$;
    \item $a\geq 2$ and $k\leq 2(a-1)e-2$.
\end{itemize}
\end{proposition}

We can now apply Proposition \ref{prop_det_of_tautological_bpf_va} to obtain the following corollaries.

\begin{corollary}\label{cor_bpf_va_line_bundle_for_k3}
Let $S$ be a K3 surface such that $\Pic(S)\cong\bZ L$ with $L$ ample and $L^2=2e$. Then the line bundle $aL_n-\delta$ on $\hilb{S}$ is base point free if either
\begin{itemize}
    \item $a=1$ and $n\leq \frac{e}{2}+1$;
    \item $a\geq 2$ and $n\leq 2(a-1)e-1$.
\end{itemize}
Also, the line bundle $aL_n-\delta$ on $\hilb{S}$ is very ample if either
\begin{itemize}
    \item $a=1$ and $n\leq \frac{e}{2}$;
    \item $a\geq 2$ and $n\leq 2(a-1)e-2$.
\end{itemize}
\end{corollary}

\begin{corollary}\label{cor_bpf_va_line_bundle_for_abelian}
Let $A$ be an Abelian surface such that $\Pic(A)\cong\bZ L$ with $L$ ample and $L^2=2e$. Then the line bundle $aL_{n+1}-\delta$ on $\Kum^n(S)$ is base point free if either
\begin{itemize}
    \item $a=1$ and $n\leq \lfloor\frac{e-3}{2}\rfloor$;
    \item $a\geq 2$ and $n\leq 2(a-1)e-2$.
\end{itemize}
Also, the line bundle $aL_{n+1}-\delta$ on $\Kum^n(S)$ is very ample if either
\begin{itemize}
    \item $a=1$ and $n\leq \lfloor\frac{e-3}{2}\rfloor-1$;
    \item $a\geq 2$ and $n\leq 2(a-1)e-3$.
\end{itemize}
\end{corollary}

\begin{remark}
    Under a little abuse of notation, we use $L_{n+1}$ to denote the line bundle on $\hilb[n+1]{A}$ and its restriction to $\Kum^n(A)$. This is not a problem since both very ampleness and base point freeness behave nicely under restriction.
\end{remark}

\begin{theorem}\label{theo_general_in_connected_component}
    Let $n,d,t$ be positive integers with $t,n\geq 2$, let $\tau =\frac{t^2}{2(t-1)}$ then
    \begin{itemize}
        \item if $d\geq (\tau-1)n+\tau+1$ there exists a connected component of $\Nu(K3)_{d,t}^n$ such that the polarization of a general element $(X,H)$ is base point free;
        \item if $d\geq (\tau-1)n+2\tau+1$ there exists a connected component of $\Nu(K3)_{d,t}^n$ such that the polarization of a general element $(X,H)$ is very ample;
        \item if $d\geq (\tau-1)n+2\tau-1$ there exists a connected component of $\Nu(\Kum)_{d,t}^n$ such that the polarization of a general element $(X,H)$ is base point free;
        \item if $d\geq (\tau-1)n+3\tau-1$ there exists a connected component of $\Nu(\Kum)_{d,t}^n$ such that the polarization of a general element $(X,H)$ is very ample;
    \end{itemize}
These hold unless the considered moduli space is empty.
\end{theorem}

\begin{proof}
    Let us focus on $\Nu(\Kum)_{d,t}^n$ in the case of very ampleness as the proofs of the other instances are very similar. Fix $t,n$ positive integers with $t,n\geq 2$ so that $t|2(n+1)$ (this is a necessary condition for $t$ to be a possible divisibility) and pick a polarized Abelian surface $(A,L)$ with $\Pic(A)\cong \bZ L$ and $L^2=e$. According to Remark \ref{remark_how_to_compute_divisibility} it then holds $$\Div(tL_{n+1}-\delta)=(t,2(n+1))=t$$
    and $$q(tL_{n+1}-\delta)=2\left(t^2e-n-1\right)=:2d.$$
By Corollary \ref{cor_bpf_va_line_bundle_for_abelian}, the line bundle $tL_{n+1}-\delta$ is very ample if $$d\geq (\tau-1)n+3\tau-1.$$
Thanks to Corollary \ref{cor_openess_of_bpf_va} and Proposition \ref{prop_irreducibility_connected_components}, the polarization of a general element $(X,H)$, in the connected component of $\Nu(\Kum)_{d,t}^n$ containing $(\Kum^n(A),tL_{n+1}-\delta)$, is very ample.
    \end{proof}

\begin{corollary}\label{corollary_main_result_for_one_connected_component}
        Let $n,d$ be positive integers with $n\geq 2$.\\
        Let $(X,H)$ be a general element of $\Nu(K3)_{d,t}^n$ and suppose $t=2$ or $t=p^a$ with $p> 2$ a prime number such that $p^{a+1}$ does not divide $(2d,2(n-1))$ then
    \begin{itemize}
        \item if $d\geq (\tau-1)n+\tau+1$ the line bundle $H$ is base point free;
        \item if $d\geq (\tau-1)n+2\tau+1$ the line bundle $H$ is very ample;
    \end{itemize}
    Let $(X,H)$ be a general element of $\Nu(\Kum)_{d,t}^n$ and suppose $t=2$ or or $t=p^a$ with $p> 2$ a prime number such that $p^{a+1}$ does not divide $(2d,2(n+1))$ then
    \begin{itemize}
        \item if $d\geq (\tau-1)n+2\tau-1$ the line bundle $H$ is base point free;
        \item if $d\geq (\tau-1)n+3\tau-1$ the line bundle $H$ is very ample;
    \end{itemize}
These hold unless the considered moduli space is empty.
\end{corollary}

\begin{proof}
    This follows from Theorem \ref{theo_count_connected_components_hilb_k3} or Theorem \ref{theo_count_connected_components_kumm_type} after checking that these moduli spaces have one connected component. For the sake of it, we are going to check that this works for $\Nu(K3)_{d,t}^n$ with $t=p^a$ power of a prime number greater than $2$ such that $p^{a+1}$ does not divide $(2d,2(n-1))$. Under these conditions, using the same notation as in Theorem \ref{theo_count_connected_components_hilb_k3}, we have that the prime $p$ does not divide $g$ hence, since $t=p^a$, it holds $w=1$ and $t_1=t$. Therefore, we have that $t_1$ is odd, hence we necessarily are in the first case of Theorem \ref{theo_count_connected_components_hilb_k3}. Thus $\rho(t_1)=\rho(t)=1$ and $w_+(t_1)=w_-(t_1)=1$ hence $|\Nu(K3)_{d,t}^n|=1$, as desired.
\end{proof}

Note that many of the moduli spaces that do not fall in the conditions of Theorem \ref{theo_general_in_connected_component} are actually empty. For example, from \ref{remark_how_to_compute_divisibility} it follows that a necessary condition for $\Sigma_{d,n}^t$ to not be empty is $t|2(n-1)$. With this in mind, we are giving a necessary and sufficient condition for $\Nu(K3)_{d,t}^n$ and $\Nu(\Kum)_{d,t}^n$ to be non empty.

\begin{proposition}\label{prop_char_of_non_emptyness}
    Let $n,d,t$ be integers with $n\geq 2$ then
    \begin{equation*}\Nu(K3)_{d,t}^n\neq \emptyset \iff 
        \begin{cases}
            t|(2n-2)\\
            d\equiv_{t^2}-b^2(n-1) \text{ for some } b \text{ such that } $(b,t)=1$
        \end{cases}
    \end{equation*}
    and 
    \begin{equation*}\Nu(\Kum)_{d,t}^n\neq \emptyset \iff 
        \begin{cases}
            t|(2n+2)\\
            d\equiv_{t^2}-b^2(n+1) \text{ for some } b \text{ such that } $(b,t)=1$
        \end{cases}
    \end{equation*}
\end{proposition}
\begin{proof}
    Suppose $\Nu(K3)_{d,t}^n$ is non-empty, then we can pick an element $(X,H)$. Thanks to Proposition \ref{p_k3_type_lattice}, we can find a K3 surface $S$ and a line bundle $L$ on $S$ such $H\cong aL+b\delta$ for some integers $a,b$. Recalling that $q(\delta)=-2(n-1)$, the line bundle $aL+b\delta$ satisfies the following numerical properties
    \begin{align*}
        a^2q(L)-2b^2(n-1)&=2d,\\
        (a,2b(n-1))&=t,\\
        (a,b)&=1.
    \end{align*}
    From the second equation, we obtain $t|a$, and since $(a,b)=1$, then $(t,b)=1$ as wanted.
    Knowing that $(t,b)=1$ allows us to conclude $t|(2n-2)$. Finally, from the first equation we can now write $$d\equiv_{t^2}-b^2(n-1).$$ 
    Suppose now that we have positive integers $d,n,t$ with $n\geq2 $ satisfying 
    \begin{equation*}
       \begin{cases}
            t|(2n-2)\\
            d\equiv_{t^2}-b^2(n-1) \text{ for some } b \text{ such that } $(b,t)=1$.
        \end{cases}
    \end{equation*}
    Observe that the number $$e:=\frac{d+b^2(n-1)}{t^2}$$ is a positive integer, so there exists a K3 surface $S$ with line bundle $L$ such that $L^2=2e$, and after deforming we can assume $\Pic(S)\equiv \bZ L$, so $L$ is primitive. Consequently,  $tL_n+b\delta$ is primitive as $(t,b)=1$, and it satisfies 
    \begin{align*}
        q(tL_n+b\delta)&=2t^2e-2b^2(n-1)=2d,\\
        \Div(tL_n+b\delta)&=(t,2b(n-1))=t.
    \end{align*}
    For general deformation $(X,H)$ of $(\hilb{S},tL_n-b\delta)$, the Picard group is generated by $H$, hence by the projectivity criteria in Theorem \ref{theo_projective_hyperkahler_manifolds}, the line bundle $H$ or $-H$ is ample. Suppose $H$ is, then $(X,H)\in \Sigma_{d,n}^t$ as wanted. If $-H$ is ample, we can conclude analogously, since $\Div(-H)=\Div(H)$ and $q(H)=q(-H)$.
    The proof for $\Nu(\Kum)_{d,t}^n$ similarly follows.
\end{proof}

\subsection{Case 2: all connected components}

In this section, the goal is to obtain an analogous of Corollary \ref{corollary_main_result_for_one_connected_component} that holds for any connected component. That is, we want to find a lower bound in terms of $d,n,t$ such that given any general element $(X,H)\in \Nu(K3)_{d,t}^n$ (or in $\Nu(\Kum)_{d,t}^n$) then $H$ is base point free, or very ample. The idea is to use the map $f:\mu_{d,n}^t\to \Sigma_{d,n}^t$ (or $f:\mu_{d,n}^t\to \Upsilon_{d,n}^t$) to obtain an element $(X,H)$ in any connected component of the moduli space such that, under suitable constraints on $d,n,t$, $H$ is base point free, or very ample. 

\begin{lemma}
    Let $S$ be a K3 surface and let $X=\hilb{S}$, then the primitive embedding $$j_X:H^2(X,\bZ)\cong H^2(S,\bZ)\oplus \delta\bZ\hookrightarrow H^{even}(S,\bZ)$$ given by $j_x(\alpha,0)=(0,\alpha,0)$ and $j_X(0,\delta)=(1,0,n-1)$ is in the distinguished $O(\widetilde{\Lambda})$-orbit $[\iota_x]$ obtained from Theorem \ref{theorem_distinguished_isometric_embedding_orbit}. 

    Let $S$ be an Abelian surface and let $X=\Kum^n(S)$, then the primitive embedding $$j_X:H^2(X,\bZ)\cong H^2(T,\bZ)\oplus \delta\bZ\hookrightarrow H^{even}(S,\bZ)$$ given by $j_x(\alpha,0)=(0,\alpha,0)$ and $j_X(0,\delta)=(1,0,n+1)$ is in the distinguished $O(\widetilde{\Lambda})$-orbit $[\iota_x]$.
\end{lemma}
\begin{proof}
    We are going to describe the proof in the $\hilb{K3}$ case as the other follows analogously. Let $v=(1,0,n-1)\in H^{even}(S,\bZ)$ then $M_H(v)$, the moduli space of stable sheaves on $S$ with Mukai vector $v$, is identified with $\hilb{S}$, (cf. \cite{sawon_moduli_space_sheaves_k3}). Following Example 9.6 in \cite{markman_survey_torelli_and_monodromy}, with the Mukai vector $v$ as above, yields that $$[\iota_{M_H(v)}]=[\theta^{-1}]$$ where $\theta: v^\perp\to H^2(M_H(v),\bZ)$ is the Mukai isometry. Using the fact that $M_H(v)$ can be identify with $X=\hilb{S}$ and $\theta^{-1}=j_X$, this concludes the proof. 

    For the $\Kum^n$ case we use the fact that for $v=(1,0,n+1)$ the moduli space $K_v(T)$ of stable sheaves on the Abelian surface $T$ is $\Kum^n(T)$. From Propositions 3.1 - 3.2 in \cite{onorati2022connectedcomponentsmodulispaces} it follows again that for $K_v(T)$ the distinguished embedding is $j_{K_v(T)}$ hence $j_X$.
\end{proof}

\begin{lemma}\label{lemma_rank_2_pos_def_basis}
    Let $T$ be an even rank 2 positive definite lattice. Let $h,v$ be primitive elements in $T$ with $(v,h)=0$. Suppose that the discriminant of $T$ is equal to $\frac{4dm}{t^2}$ where $v^2=2m>0$, $h^2=2d>0$ and $t$ is a positive integer. Then
    \begin{itemize}
        \item there exists $b\in \bZ$ with $(b,t)=1$ and, if $t\geq2$, with $1\leq b\leq\lfloor\frac{t}{2}\rfloor$ such that $T=\bZ u\oplus \bZ v$, where $u$ satisfies $h=tu-bv$;
        The Gram matrix of $T$ in this basis is 
        $$\begin{pmatrix}
            2e_b&\frac{2mb}{t}\\
            \frac{2mb}{t}&2m
        \end{pmatrix}$$
        where $e_b:=\frac{d+mb^2}{t^2}\in\bZ_{>0}$.
        \item If $T$ is primitively embedded in a unimodular lattice $\widetilde{\Lambda}$, then $\Div_{v^\perp}(h)=t$.
    \end{itemize}
\end{lemma}

\begin{proof}

Let $T':=\bZ h\oplus \bZ v$ then in the base $h,v$ we have $\disc(T')=4md$. Furthermore, since $T'$ has rank 2, it has finite index in $T$, so the following formula holds $$\disc(T')=[T:T']\disc(T).$$
Substituting the discriminants we obtain $[T:T']=t$.

Denote by $\pi:T\to T/\bZ v$ the quotient map. Since $v$ is primitive in $T$, we have $T/\bZ v\cong \bZ$. Pick $u\in T$ such that $\pi(u)$ generates $T/\bZ v$. For any element $x\in T$ there exists $a\in\bZ$ such that $\pi(x)=a\pi(u)$ hence $x-au\in \bZ  v$, from which we can conclude $T=\bZ u\oplus \bZ v$. Let $a,b\in\bZ$ such that $h=au+bv$, then the inclusion matrix of $T'$ in $T$ with respect to the basis $h,v$ of $T'$ and $u,v$ of $T$ is $$\begin{pmatrix}
            a&0\\
            b&1
        \end{pmatrix}.$$
A standard result in lattice theory states that the index $[T:T']$ is then equal to the absolute value of the determinant of the inclusion matrix, hence $[T:T']=|a|$. However, we already know that $[T:T']=t$, hence $|a|=t$ and, up to choosing $-u$ instead, we have $a=t$.
\end{proof}

The idea now is to show that in any connected component of $\Nu(K3)_{d,t}^n$ we can find an element of the form $(\hilb{S},H)$, where $S$ is a $K3$ surface with Picard rank one.

\begin{lemma}\label{lemma_X_H_with_prescribed_T(X,H)}
    Given $[(T,h)]\in\Sigma_{d,t}^n$, we can find a $K3$ surface $S$ with $\Pic(S)=\bZ L$ and, if we let $X=\hilb{S}$, there exists $H\in \Pic(X)$ such that $[(T(X,H),j_X(H))]=[(T,h)]$.
    Analogously, Given $[(T,h)]\in\Upsilon_{d,t}^n$, we can find an Abelian surface $A$ with $\Pic(A)=\bZ L$ and, if we let $X=\Kum^n(A)$, there exists $H\in \Pic(X)$ such that $[(T(X,H),j_X(H))]=[(T,h)]$.
\end{lemma}

 \begin{proof}
     Given $[(T,h)]\in\Sigma_{d,t}^n$ or in $\Upsilon_{d,t}^n$, Lemma \ref{lemma_rank_2_pos_def_basis} allows us to obtain a basis $u,v$, with the properties described in the Lemma.

     Consider now a K3 surface $S$ with $\Pic(S)\cong \bZ L$ and let $X=\hilb{S}$. Denote by $H:=tL_n-b\delta$, where $b$ is obtained from Lemma \ref{lemma_rank_2_pos_def_basis}, then $j_X(H)=\left(-b,tL,-bm\right)$, recalling that $j_X(L)=(0,L,0)$ and $j_X(\delta)=(1,0,-m)$ where $m=n-1$. To simplify the notation, denote $v_0=j_X(\delta)$, $h_0=j_X(H)$ and let $u_0:=(0,L,-\frac{2bm}{t})$, we claim $T(X,H)=\bZ u_0\oplus \bZ v_0$. Recall that $T(X,H)$ is the saturation of the lattice $\bZ h_0\oplus (j_X(H^2(X,\bZ)))^\perp=\bZ h_0\oplus \bZ v_0$ in $\widetilde{\Lambda}$. It holds $T(X,H)\otimes \bQ=\bQ u_0\oplus \bQ v_0$ hence if we prove that $\bZ u_0\oplus \bZ v_0$ is primitive in $\widetilde{\Lambda}$ we are done. Let $x=ru_0+sv_0$ with $r,s\in \bQ$ such that $x\in \widetilde{\Lambda}$. Then $x=(s,rL,-sm-\frac{2rbm}{t})$ and primitivity of $L$ forces $r$ to be an integer, while the first component must be in $H^0(S,\bZ)$, thus $s$ is also an integer.

     To conclude, define the map $\phi:(T,h)\to(T(H,X),j_X(H))$ by $\phi(u)=u_0$, $\phi(v)=v_0$. It is an immediate check that $\phi(h)=j_X(H)$ and that $\phi$ is indeed an isometry.
 \end{proof}

 \begin{remark}
     The element $(X,H)$ defined in the previous proof does not necessarily belong to $\Nu(K3)_{d,t}^n$ for any $d,t$ as $H$, constructed above, does not have to be ample. 
 \end{remark}

\begin{theorem}\label{theorem_main_all_components}
     Let $n,d,t$ be positive integers with $n,t\geq 2$, let $\tau=\frac{t^2}{2(t-1)}$ and recall $e_i=\frac{d+mi^2}{t^2}$.
     
     For the $\hilb{K3}$-type we have the following.\\
     Under the assumption $n\leq \frac{e_1}{2}+1$:
\begin{itemize}
    \item if $d\geq (\tau-1)n+\tau+1$ then the polarization of a general element in any connected component of $\Nu_{d,t}^n(K3)$ is base point free, and the $n+2$ power of the polarization is very ample;
    \item if $d\geq (\tau-1)n+2\tau+1$ then the polarization of a general element in any connected component of $\Nu_{d,t}^n(K3)$ is very ample.
\end{itemize}   
    Otherwise,
    \begin{itemize}
        \item if $d\geq \left(\frac{t^2}{2}-1\right)n+\frac{t^2}{2}+1$ then the polarization of a general element in any connected component of $\Nu_{d,t}^n(K3)$ is base point free, and the $n+2$ power of the polarization is very ample;
        \item if $d\geq \left(\frac{t^2}{2}-1\right)n+t^2+1$ then the polarization of a general element in any connected component of $\Nu_{d,t}^n(K3)$ is very ample.
    \end{itemize}
    
For the $\Kum^n$-type we have the following.\\
Under the assumption $n\leq \frac{e_1-3}{2}$:
\begin{itemize}
    \item if $d\geq (\tau-1)n+2\tau-1$ then the polarization of a general element in any connected component of $\Nu_{d,t}^n(\Kum)$ is base point free and the $n+2$ power of the polarization is very ample;
    \item if $d\geq (\tau-1)n+3\tau-1$ then the polarization of a general element in any connected component of $\Nu_{d,t}^n(\Kum)$ is very ample.
\end{itemize}
Otherwise,
\begin{itemize}
    \item if $d\geq \left(\frac{t^2}{2}-1\right)n+t^2-1$ then the polarization of a general element in any connected component of $\Nu_{d,t}^n(\Kum)$ is base point free, and the $n+2$ power of the polarization is very ample;
    \item if $d\geq \left(\frac{t^2}{2}-1\right)n+\frac{3t^2}{2}-1$ then the polarization of a general element in any connected component of $\Nu_{d,t}^n(\Kum)$ is very ample.
\end{itemize}
\end{theorem}

\begin{proof}
    The proof of all the cases follows a similar structure, so we are going to prove the theorem for the base point freeness condition in the $\hilb{K3}$-type. 
    
    Fix a connected component of $\Nu_{d,t}^n(K3)$ and pick $[(T,h)]\in \Sigma_{d,t}^n$ that identifies the chosen connected component. From Lemma \ref{lemma_X_H_with_prescribed_T(X,H)} we obtain $(X,H)$ such that $[(T,h)]=[(T(X,H),\iota_X(H))]$. Recall $X=\hilb{S}$ for a suitable $K3$ surface $S$ satisfying $\Pic(S)=\bZ L$, and $H=tL_n-b\delta$. 

    Assume now that $n\leq \frac{e_1}{2}+1$. In this situation, we can write $$H=(b-1)(L_n-\delta)+(t-b+1)L_n-\delta$$ and, since $n\leq \frac{e_1}{2}+1$, Corollary \ref{cor_bpf_va_line_bundle_for_k3} implies that $L_n-\delta$ is base point free. The same corollary also implies that $(t-b+1)L_n-\delta$ is base point free if $d\geq (\tau-1)n+2\tau-1$, where we used the fact that $b\geq 1$. Theorem 1.1 in \cite{hinohara_takahashi_terakawa_tensor_product_k_very_ample} let us conclude that $H$ is base point free.

    Suppose instead that we have no inequality on $n$. Recall that $b\leq\frac{t}{2}$, hence there exists $a_1,\dots,a_b\geq 2$ such that $t=a_1+\dots+a_b$. Applying Corollary \ref{cor_bpf_va_line_bundle_for_k3} the line bundles $a_iL_n-\delta$ are base point free if $d\geq \frac{n+2}{2}t^2-n-1$. Under this condition, we then find that $H$ is also base point free. 

    In both case we end up with $(X,H)$ where $q(H)=2d$, $\Div(H)=t$, $H$ is base point free and $[(T,h)]=[(T(X,H),\iota_X(H))]$. Let $(X',H')$ be a small deformation of $(X,H)$ with $H'$ ample. The ample line bundle $H'$ satisfies $q(H')=2d$, $\Div(H)=t$ and $H'$ is base point free thanks to Corollary \ref{cor_openess_of_bpf_va}. The only thing to check is that $(X',H')$ belongs to the connected component of $\Nu_{d,t}^n(K3)$ prescribed by $(T,h)$. A parallel transport operator between $(X,H)$ and $(X',H')$ gives an isometry $g:H^2(X,\bZ)\to H^2(X',\bZ)$ such that $g(H)=H'$ and $[\iota_X]=[\iota_{X'}\circ g]$ (Theorem 9.8 in \cite{markman_survey_torelli_and_monodromy}), hence $[(T,h)]=[T(X',H'),\iota_{X'}(H')]$. We can conclude by applying Lemma \ref{lemma_bfp_does_not_require_ampleness}.
\end{proof}

\begin{remark}
    The analogous property for parallel transport in the $\Kum^n$-case can be obtained from Proposition 3.3 in \cite{onorati2022connectedcomponentsmodulispaces}.
\end{remark}

We would like to conclude with one more remark.\\
Recall the following notion from \cite{Lazarsfeld2004PositivityI} (Definition 1.8.4).

Let $X$ be a projective variety and $B$ an ample line bundle on $X$ that is generated by its global sections. A coherent sheaf $\mathcal{F}$ on $X$ is \emph{$m$-regular with respect to $B$} if $$H^i(X,\mathcal{F}\otimes B^{\otimes(m-i)})=0$$ for $i>0$.

What we are interested in is the following proposition:
\begin{proposition}[\cite{Lazarsfeld2004PositivityI}, Prop 1.8.22]\label{prop_0_regular_gives_very_ample}
Let $X$ be an irreducible projective variety of dimension $n$ and $B$ an ample line bundle that is generated by its global sections and let $N$ be a line bundle that is 0-regular with respect to $B$. Then $N\otimes B$ is very ample.
\end{proposition} 

From which we obtain the following lemma.

\begin{lemma}
    Let $X$ be a hyperk\"ahler manifold of complex dimension $n$ and let $L$ be an ample base point free line bundle on $X$. Then $L^{n+2}$ is very ample.
\end{lemma}

\begin{proof}
    If we can apply Proposition \ref{prop_0_regular_gives_very_ample} with $B=L$ and $N=L^{n+1}$, then $N\otimes B=L^{n+2}$ is very ample and we are done. We need to show that $$H^i(X,N\otimes B^{\otimes(0-i)})=0\quad \forall i>0$$
    which, with our choices of $N$ and $B$, becomes $$H^i(X,L^{\otimes(n+1-i)})=0\quad \forall i>0.$$
    If $i<n+1$, then Kodaira vanishing can be applied.
    If $i\geq n+1$, then the cohomology group is trivial since $h^{k,l}(X,H)=0$ for any line bundle $H\in \Pic(X)$, if $k>\dim_{\mathbb{C}} X=n$.
\end{proof}

\begin{corollary}
     Under the assumption of \ref{theorem_main_all_components} if the polarization of a general element $(X,H)$ is base point free then $H^{n+2}$ is very ample.
\end{corollary}


{\small
\printbibliography}

\noindent {\sc Alessandro Pilastro\\
Department of Mathematics\\
Stony Brook University \\
Stony Brook, NY 11794-3651, USA,} \\
\tt alessandro.pilastro@stonybrook.edu\\

\end{document}